\documentclass{amsart}
\usepackage{amsfonts,amscd}

\newtheorem{theorem}{Theorem}[section]
\newtheorem{lemma}[theorem]{Lemma}
\newtheorem{corollary}[theorem]{Corollary}
\newtheorem{proposition}[theorem]{Proposition}
\theoremstyle{remark}

\theoremstyle{definition}
\newtheorem{definition}[theorem]{Definition}
\newtheorem{example}[theorem]{Example}
\numberwithin{equation}{section} \makeatother

\DeclareMathOperator{\Cdb}{{\mathbb C}}
\DeclareMathOperator{\Rdb}{{\mathbb R}}

\DeclareMathOperator{\Ndb}{{\mathbb N}}

\begin{document}
\title[Characterizations of
unital operator spaces and systems]{Metric characterizations of isometries and
of unital operator spaces and systems}

\author{David P. Blecher}
\address{Department of Mathematics, University of Houston, Houston, TX
77204-3008}
\email[David P. Blecher]{dblecher@math.uh.edu}
 \author{Matthew Neal}
\address{Department of Mathematics,
Denison University, Granville, OH 43023}
\email{nealm@denison.edu}
\date{May 13, 2008.  Revision of May 22, 2008.}
\thanks{*Blecher was partially supported by grant DMS 0800674 from
the National Science Foundation. Neal was supported by Denison
University.}

\begin{abstract}
We give some new characterizations of unitaries, isometries, unital
operator spaces, unital function spaces, operator systems,
$C^*$-algebras, and related objects.   These characterizations only
employ the vector space and operator space structure.
\end{abstract}

\maketitle

\section{Introduction}

We give some new characterizations, of unitaries, isometries, unital
operator spaces, unital function spaces, operator systems,
$C^*$-algebras,  and related objects.  Several characterizations of
these objects are already known; see e.g.\
\cite[Theorem 9.5.16]{Pal}, \cite{AW}, the discussion on p.\
316 of \cite{BEZ}, \cite{NR}, and \cite{HN}.
One difference between our paper and these cited references, is that
our results only use the vector space structure
of the space and its matrix norms, in the spirit of Ruan's matrix norm
characterization of operator spaces \cite{Ru}, whereas the other
cited references
use criteria involving maps or functionals on the space.
   Our first main
result characterizes unital operator spaces, that is, subspaces of a
unital $C^*$-algebra containing the identity.
More abstractly, a
unital operator space is a pair $(X,u)$ consisting of an operator
space $X$ containing a fixed element $u$ such that there exists a
Hilbert space $H$ and a complete isometry $T : X \to B(H)$ with
$T(u) = I_H$. Such spaces have played a significant role since the
birth of operator space theory in \cite{Ar}.  Indeed, although
 the latter paper is mostly concerned with
unital operator algebras, it was remarked in several places there
that many of the results are valid for unital operator spaces.  The
text \cite{BLM} also greatly emphasizes unital operator spaces. The
abstract characterization of these objects has been missing until
recently, and we had wondered about this over the years; the
following is our answer to this question. Our result complements
Ruan's characterization of operator spaces \cite{Ru}, the
Blecher-Ruan-Sinclair abstract characterization of operator algebras
\cite{BRS}, and a host of other theorems of this type (see e.g.\
\cite{BLM,Pau}). To state it, we will write $u_n$ for the diagonal
matrix in $M_n(X)$ with $u$ in each diagonal entry.

\begin{theorem} \label{main}  If $u$ is an element in an operator space  $X$,
then $(X,u)$ is a unital operator space if and only if
$$\Vert [ u_n  \; \; x ] \Vert = \left| \left|   \left[ \begin{array}{cl} u_n
 \\ x \end{array} \right]
 \right| \right| = \sqrt{2} ,$$
for all $x \in M_n(X)$ of norm $1$, and all $n \in \Ndb$.
 \end{theorem}

We call the element $u$ in the last theorem  a {\em unitary in} $X$,
or a {\em unit} or {\em identity} for $X$.  We also find the
matching abstract characterization of function spaces containing
constants.  Namely a pair $(X,g)$, where $X$ is now a Banach space,
is  a unital function space iff $\sup \{ \Vert s f + t g \Vert : s,
t \in \Cdb, |s|^2 + |t|^2 = 1 \} = \sqrt{2}$ for every $f \in X$
with $\Vert f  \Vert = 1$.    This is presented in Section 5.
 In section 2 of our paper we also characterize unitaries and isometries
in a $C^*$-algebra, etc. In Section 3 we characterize and study
operator systems, that is selfadjoint subspaces of a unital
$C^*$-algebra containing the identity.   More abstractly, an
operator system is a unital operator space $(X,u)$ for which there
exists a linear complete isometry $T : X \to B(H)$ with $T(u) = I_H$
and $T(X)$ selfadjoint. Theorem \ref{main} leads to new intrinsic
characterizations of operator systems.
For example:

\begin{theorem} \label{main3}  A unital operator space $(X,u)$
(characterized above) is an operator system iff for all $x \in {\rm
Ball}(X)$
there exists an element $y \in {\rm Ball}(X)$ with
$\left| \left|   \left[ \begin{array}{ccl} tu & x \\ y & tu
\end{array} \right]
 \right| \right|   \leq \sqrt{t^2 + 1}$ for all $t \in \Rdb$.
  It is also equivalent to:
 $$\inf \Bigl\{ \left| \left|
 \left[ \begin{array}{ccl} tu & x \\ y & tu
 \end{array} \right]
 \right| \right|
  \; : \; y \in {\rm
Ball}(X) \Bigr\} \leq \sqrt{t^2 + 1}$$  for all $t \in \Rdb$
and $x \in {\rm
Ball}(X)$.
\end{theorem}


Similarly, we obtain new `matrix norm' characterizations of
$C^*$-algebras. For example:

\begin{theorem} \label{main4prime}  A unital operator space $(X,1)$
(characterized above) possesses a product with respect to which it
is isomorphic to a $C^*$-algebra via a unital complete isometry, if
and only if $X$ is spanned by the unitaries in $X$ (characterized in
Section {\rm 2}) and  for every unitary $v$ in $X$ we have
$$\inf \Bigl\{ \left| \left|
 \left[ \begin{array}{ccl} t1 & y \\ z & tv
 \end{array} \right]
 \right| \right|
  \; : \; z \in {\rm
Ball}(X) \Bigr\} \leq \sqrt{t^2 + 1}$$  for all $t \in \Rdb$ and $y
\in {\rm Ball}(X)$.
 \end{theorem}



 Section 4 discusses changing the identity in an operator system;
also we mention a  connection between
 our theory and the famous
characterization due to Choi and Effros of operator systems, in
terms of an order unit \cite{CE,Pau}.

We now turn to precise definitions and notation.    Any unexplained
terms below can be found in \cite{BLM}, or any of the other recent
books on operator spaces.  All vector spaces are over the complex
field $\Cdb$.  The letters $H, K$ are usually reserved for Hilbert
spaces.  A given cone in a space $X$ will sometimes be written as
$X_+$, and
 $X_{\rm sa} = \{ x \in X : x = x^* \}$ assuming that there
is an involution $*$ around.  All normed (or operator) spaces are
assumed to be complete.
 A (resp.\ complete) {\em order isomorphism} is a (resp.\
 completely) positive linear bijection $T$ such that
$T^{-1}$ is (resp.\ completely) positive.  It is well known that a
surjective complete isometry $T$ between operator systems with $T(1)
= 1$ is a complete order isomorphism (see e.g.\ 1.3.3 in
\cite{BLM}). Thus we will not be too concerned with
positivity issues in this paper.

A {\em TRO} (ternary ring of
operators) is a closed
subspace $Z$ of a C*-algebra, or of $B(K,H)$,  such that $Z Z^* Z \subset Z$.
We refer to e.g.\ \cite{Ham,BLM} for the basic theory of TROs.
A {\em ternary morphism} on a TRO $Z$ is a linear map $T$ such that
$T(x y^* z) = T(x) T(y)^* T(z)$ for all $x, y, z \in Z$.
We write $Z Z^*$ for the closure of the linear span of
products $z w^*$ with $z, w \in Z$, and similarly for $Z^* Z$. These are
$C^*$-algebras.
 The {\em ternary envelope}
of an operator space  $X$ is a pair $({\mathcal T}(X),j)$ consisting of
a TRO ${\mathcal T}(X)$ and a  completely isometric linear map $j : X
\to {\mathcal T}(X)$, such that
${\mathcal T}(X)$ is generated by $j(X)$ as a TRO (that is, there is no closed
subTRO containing $j(X)$), and which has the following
property: given any completely isometric
linear map $i$ from $X$ into a TRO $Z$ which is
generated by $i(X)$, there exists a (necessarily unique and surjective)
ternary morphism $\theta : Z \to {\mathcal T}(X)$ such that
$\theta \circ i = j$.
If $(X,u)$ is a unital operator space then the ternary envelope
may be taken to be the {\em $C^*$-envelope} of e.g.\ \cite[Section 4.3]{BLM};
this is a $C^*$-algebra $C^*_e(X)$ with identity $u$.
If $X$ is an operator system then $X$ is  a selfadjoint subspace
of $C^*_e(X)$.

We remark that the criteria
appearing in the characterizations in \cite{HN}
have nothing in common with those in our results,  nor
do the methods of proof.
For example,
their criteria for unital operator spaces or systems
are in terms of completely contractive unital matrix valued
maps $\varphi$ on the space, and
they rely on there being `sufficiently many' of such maps.

\section{Characterization of isometries and unital spaces}

Clearly, the definition of a unital operator space $(X,u)$ above is
unchanged if we replace $B(H)$ by a unital $C^*$-algebra, or if we
replace $I_H$ with any unitary.   Thus the element $u$ is called a
unitary in $X$.  Similarly, we say that an element $v$ in an
operator space $X$ is an {\em isometry} (resp.\ {\em coisometry}) {\em in}
$X$ if there exists a  complete isometry $T$ from $X$ into
 $B(K,H)$, for Hilbert spaces $H$ and $K$,
 with $T(v)$ an isometry (resp.\ coisometry).

A {\em unitary} in a TRO $Z$ is an element $u \in Z$ with $u u^* z =
z$ and $z u^* u = z$ for all $z \in Z$. We say that $u$ is an {\em
coisometry} (resp.\ {\em isometry})
 if just the first (resp.\ second) condition holds.
If $Z$ is  a $C^*$-algebra it is easy to see that these coincide with
the usual definition of unitary, coisometry, and isometry.
We will  soon see that these also coincide with the operator space
definitions above.

\begin{theorem} \label{utro}  An element $u$ in a
$C^*$-algebra or TRO $A$ is a unitary if and only if $\Vert [ u \;
\; x ] \Vert^2 = 1 + \Vert x \Vert^2$ and  $\Vert [ u \; \;  x ]^t
\Vert^2 = 1 + \Vert x \Vert^2$, for all $x \in A$.  Indeed, it
suffices to consider norm one elements $x$ here. Similarly, $u$ is a
coisometry (resp.\ isometry) iff the first (resp.\ second) of these
norm conditions holds for all $x \in A$.
 \end{theorem}

\begin{proof}   We just prove the coisometry assertion, the others
following by symmetry.
By the $C^*$-identity, $\Vert [ u \; \; x ] \Vert^2
 = \Vert u u^* + x x^* \Vert$.  If $u$ is a
coisometry then this equals $1 + \Vert x \Vert^2$. Conversely,
suppose that  $\Vert [ u \; \; x ] \Vert^2 = 2$ for all $x \in A$ of
norm $1$.   It is easy to see that this implies $\Vert u \Vert = 1$.
By the $C^*$-identity,  the norm of the matrix with entries $u^* u,
u^* x, x^* u, x^* x$, is also $2$.  By writing this matrix as a
diagonal matrix plus another matrix, the last norm is $\leq \max \{
\Vert u^* u \Vert , \Vert x^* x  \Vert \} + \Vert u^* x \Vert = 1 +
\Vert u^* x \Vert \leq  2$. Hence $ \Vert u^* x \Vert  = 1$. That $u
u^* x = x$ can be deduced now from the well known functional
calculus in JB$^*$-triples (see e.g.\ p.\ 238 in \cite{BN1} for an
exposition of the TRO case of this), which shows that $u$ is a
partial isometry, hence $\Vert u u^* x - x \Vert = 0$. Instead we
will use some well known and elementary facts from $C^*$-module
theory. By the above, the operator $L_{u^*}$ of left multiplication
by $u^*$ is an isometric right module map from $A$ onto the closed
right ideal (submodule) $u^* A$ of $A^* A$. However every isometric
$C^*$-module map is `unitary' (see e.g.\ \cite[Corollary
8.1.8]{BLM}). That is, in $C^*$-module language,
$$\langle u u^* x , x \rangle = \langle u^* x , u^* x \rangle =
\langle x ,  x \rangle = x^* x , \qquad x \in A . $$ Since $u u^*
\leq 1$, it follows easily that $\Vert (1 - u u^*) x \Vert^2 = 0$ as
desired.
     \end{proof}

\begin{lemma} \label{note}
Let $u$ be an element in an operator space $X$.
The following are equivalent:
\begin{itemize}
\item [(i)]  $u$ is a unitary (resp.\ isometry, coisometry) in $X$.
\item [(ii)]  There exists a
TRO $Z$ containing $X$  completely  isometrically,
such that $u$ is a unitary (resp.\ isometry, coisometry)
in $Z$.
\item [(iii)]  The image of $u$ in the ternary envelope ${\mathcal T}(X)$
is a unitary (resp.\ isometry, coisometry).
\end{itemize}
 \end{lemma}

\begin{proof}  We focus mainly the coisometry case, the others
usually being similar.   That (iii) implies (ii) is obvious.

(i) $\Rightarrow$ (ii) \ If $v = T(u)$ is a coisometry for a
complete isometry $T : X \to B(K,H)$, then $v$ is
a coisometry in the TRO $Z$ generated by $T(X)$ in $B(K,H)$.

(ii) $\Rightarrow$ (iii) \ We may assume that $Z$ is
generated by $X$, and then this follows by the universal
property of ${\mathcal T}(X)$.

(ii) $\Rightarrow$ (i) \ If $u$ is a coisometry in
$Z$ then $u u^*$ is the identity of the
 $C^*$-algebra $Z Z^*$.  If we represent the
`linking algebra' of $Z$ nondegenerately on a Hilbert space
$H \oplus K$ in the usual way (see e.g.\ 8.2.22 in \cite{BLM}),
then $u u^* = I_H$, so that $u$ is a coisometry from $K$ to $H$.

Note that if $u$ is a  unitary in $Z = {\mathcal T}(X)$,
then $u^* Z$ is a $C^*$-algebra with identity $u$,
and $T x = u^* x$ is a `unital complete isometry'.
\end{proof}

One may also phrase (iii) above in terms of the injective envelope
$I(X)$.


We recall  $u_n$ is the diagonal matrix
in $M_n(X)$ with $u$ in every diagonal entry.

\begin{theorem} \label{uosp}  An element $u$ in an
operator space  $X$ is a  unitary in $X$ if and only if $\Vert [ u_n
\; \; x ] \Vert^2 = 1 + \Vert x \Vert^2$ and  $\Vert [ u_n \; \; x
]^t \Vert^2 = 1 + \Vert x \Vert^2$, for all $x \in M_n(X)$ and $n
\in \Ndb$.
 Indeed, it suffices to consider norm one
matrices $x$ here.
Similarly, $u$ is a coisometry (resp.\ isometry)
in $X$ iff the first (resp.\ second) of these norm conditions
holds for all $x \in M_n(X)$.
 \end{theorem}

\begin{proof}   We just prove the coisometry assertion, the others
following by symmetry.  The `easy  direction'
is just as in Theorem \ref{utro}.
Conversely, given
 the  $\Vert [ u_n  \; x ] \Vert^2 = 1 + \Vert x \Vert^2$
condition, we consider $X \subset Z =
{\mathcal T}(X)$.  As in the proof of Theorem \ref{utro} we
obtain
 $\Vert u^* x \Vert = \Vert x \Vert$, for all
$x \in X$.  This is equivalent to $\Vert u u^* x \Vert = \Vert x
\Vert$ for every $x \in X$.
 Indeed if $\Vert u^* x \Vert = 1$ for $x$ of norm $1$, then  $1 \geq
\Vert u u^* x \Vert \geq \Vert x^* u u^* x \Vert = 1$ by the
$C^*$-identity. Thus, the operator $L : {\mathcal T}(X) \to
{\mathcal T}(X)$ of left multiplication by $u u^*$ is an isometry on
$X$.
Similarly, it is a complete isometry on $X$.  Thus by the `essential'
property of the ternary envelope (see e.g.\ \cite{Ham} or
\cite[8.3.12 (3) and  4.3.6]{BLM}), $L$ is
an isometry on ${\mathcal T}(X)$.  By the proof of
Theorem \ref{utro}, $u$ is a coisometry in ${\mathcal T}(X)$.
By Lemma \ref{note}, $u$ is a coisometry in $X$.
\end{proof}




{\bf Remark.} \ Notice that in a selfadjoint operator space, if $u =
u^*$ then the one condition  $\Vert [ u \; \; x ] \Vert^2 = 1 +
\Vert x \Vert^2$ in the characterization of unitaries above
 is equivalent to the other condition  $\Vert [ u \; \; x ]^t \Vert^2 = 1 + \Vert x \Vert^2$,
and similarly for the matricial version of these equalities.

\section{Operator systems}

It is explained in 1.3.7 of \cite{BLM} (relying on
results from \cite{Ar}), that every unital operator space $(X,u)$
contains a canonical operator system $\Delta(X)$.  Since this system depends
(only) on the unit $u$, we will write it as $\Delta^u$.
If $X$ is represented as a subspace of $B(H)$ via a complete isometry $T$
 taking $u$ to $I_H$, then $\Delta^u = X \cap X^*$, the latter involution
and intersection taken in $B(H)$.  However the important point for us
is that as a subspace of $X$, $\Delta^u$ does not depend on
the particular $H$ or $T$.   Nor does its positive
cone, which  will be written as $\Delta^u_+$, nor does its involution;
these depend only on the unit $u$.   We now mention a recipe for describing
these elements more explicitly in terms of the norm and linear structure
of $X$.


\begin{definition}  Let $u, x$ be elements of a Banach space $X$.
 We say that $x$ is {\em $u$-hermitian} if there is a
constant $K$ such that $\Vert u + i t x \Vert^2
\leq 1 + K t^2$ for all $t \in \Rdb$.
We say that $x$ is {\em $u$-positive} if it is $u$-hermitian,
and if $\Vert \Vert x \Vert u  - x \Vert \leq \Vert x \Vert$.
\end{definition}

It is well known (and is an easy exercise) in the theory of
numerical ranges in Banach algebras, that an element $x$ in a
$C^*$-algebra with identity $u$, is selfadjoint iff it is
$u$-hermitian, and indeed in this case $\Vert u + i t x \Vert^2 = 1
+ t^2 \Vert  x \Vert^2$ for all $t \in \Rdb$.  If the reader prefers they
may take the latter as the definition of $u$-hermitian in what
follows. We remark that the given definition has the advantage that
any contractive operator  $T$ on $X$ takes $u$-hermitians to
$T(u)$-hermitians, and takes $u$-positives in Ball$(X)$  to
$T(u)$-positive elements.  Indeed the $u$-positives in Ball$(X)$ are
the $u$-hermitians with $\Vert u  - x \Vert \leq 1$.

\medskip

\begin{proposition} \label{uospa}  If $(X,u)$ is a unital operator space
then $\Delta^u$ is the span of the $u$-hermitians in $X$, and
$\Delta^u_{sa}$ is the set of $u$-hermitians in $X$, and
$\Delta^u_+$ is the set of $u$-positives in $X$.  Moreover we have
the following matricial characterization of  $\Delta^u_{sa}$
and $\Delta^u_+$: if $x
\in {\rm Ball}(X)$ then $x \in \Delta^u_{sa}$ iff $\left| \left|
\left[
\begin{array}{ccl} t u & x \\ -x &  t u
\end{array} \right]
 \right| \right| \leq \sqrt{t^2 + 1}$ for all $t \in \Rdb$; also
 $x \in \Delta^u_+$ iff $\left| \left|
\left[
\begin{array}{ccl} t u & u-x \\ x - u &  t u
\end{array} \right]
 \right| \right| \leq \sqrt{t^2 + 1}$ for all $t \in \Rdb$.
\end{proposition}

\begin{proof}  Only the last two `iff's need proof.  The first
of these is a special case of Lemma \ref{mult} below.  For the
second note that the norm condition implies that $x-u \in {\rm
Ball}(X)$, and then the first `iff' implies $x \in \Delta^u_{sa}$,
so that $x$ is $u$-positive.
\end{proof}

 Of course the involution on
$\Delta^u$ is just $(h + i k)^* = h - ik$, if $h, k \in
\Delta^u_{sa}$.

\begin{corollary}
 \label{elm}  A unital operator space $(X,u)$ is an operator system
iff the $u$-hermitians span $X$, and iff  the
$u$-positives span $X$.
\end{corollary}

{\bf Remark.}  We can obviously  replace the spans in the last result by
linear combinations of two or four elements respectively.  So,
a unital operator space $(X,u)$ is an operator system iff for every $x \in X$ there exists
a $y \in X$ with $x + y$ and $i(x-y)$ $u$-hermitian.

\begin{corollary}  Let $(X,u)$ be a unital operator space which also
possesses a conjugate linear involution $*$.  Then $(X,u)$ is an
operator system whose involution is $*$
if and only if $x = x^* \in X$ implies that $x$ is $u$-hermitian.
  \end{corollary}

\begin{proof}   Suppose that   $x = x^*$ implies that $x$ is $u$-hermitian.
Since the set of elements with $x = x^*$ spans $X$, so does the
set of $u$-hermitians.
Thus $X$ is an operator system.
The rest is obvious.
\end{proof}

{\bf Remark.}  The conditions in the last corollary are also equivalent to 
the matricial condition $\left| \left|
\left[
\begin{array}{ccl} t u & x \\ - x^* &  t u
\end{array} \right]
 \right| \right| \leq \sqrt{t^2 + 1}$ for all $t \in \Rdb$.   This follows from 
the next proof, or from Lemma
\ref{mult} below.

\medskip  

We now prove Theorem \ref{main3}.

\begin{proof} \  (Of {\bf Theorem
\ref{main3}}) \  If $I_H \in X = X^* \subset B(H)$ then by the
$C^*$-identity it is easy to see that $\left| \left|   \left[
\begin{array}{ccl} t I & x \\ -x^* & tI
\end{array} \right]
 \right| \right|^2   = t^2 + 1$ for every $t \in \Rdb$
and $x \in X$ with $\Vert x \Vert = 1$.
 Conversely, if $I_H \in X  \subset B(H)$ and
 if the condition in Theorem
\ref{main3} involving the infimum holds,
 then for all $n \in \Ndb$ there is an element
 $y_n \in X$ with
 $\left| \left|   \left[
\begin{array}{ccl}  n I & x \\ y_n &  n I
\end{array} \right]
 \right| \right|^2   \leq  n^2 + 1 + \frac{1}{n}$.
The norm of the last matrix is the unchanged if we multiply
the `diagonal entries' by $-1$.   Using the
 $C^*$-identity it follows that for every state $\varphi$ on $M_2(B(H))$ we
 have
 $$n^2 + \varphi \Bigl( \left[
 \begin{array}{ccl} x x^* & 0 \\ 0 & y_n y_n^*
 \end{array} \right] \Bigr) \pm n \varphi \Bigl( \left[ \begin{array}{ccl} 0 & x  +
 y_n^*\\ y_n + x^*  & 0
 \end{array} \right] \Bigr) \leq n^2 + 1 + \frac{1}{n} .$$
Taking
the supremum over all states $\varphi$, we deduce that $$\Vert y_n +
x^* \Vert \; = \; \left| \left| \left[
 \begin{array}{ccl} 0 & x  + y_n^* \\ y_n + x^*  & 0
 \end{array} \right] \right| \right|  \; \leq \;  \frac{1}{n}  + \frac{1}{n^2}  .$$
 Hence $y_n \to -x^*$, and so $X^* = X$.
 \end{proof}

{\bf Remark.} \ In Theorem \ref{main3} it does not suffice to take
$t = 1$, even in the case that $\Vert x \Vert = 1$.  Indeed it is
easy to argue that any nonselfadjoint unital function space will be
a counterexample to this.

\medskip

We can use the ideas above to characterize $C^*$-algebras among the
operator systems or unital operator spaces.   We will write the
identity $u$ of our operator system $X$ as $1$.   This topic is very
closely related to the question of recovering a forgotten product on
a $C^*$-algebra, which was discussed e.g.\ on p.\ 316 of \cite{BEZ}.
The route we take here is that since unitaries have been
characterized in Section 2, to characterize $C^*$-algebras it
suffices to 1) \ characterize when $X$ is the span of the unitaries
it contains, and 2) \ to characterize when the product in $C^*_{\rm
e}(X)$ of every two unitaries $u$ and $v$ in $X$ is again in $X$.
There seem to be many simple characterizations for 1) in the context
of unital $C^*$-algebras, for example that for every $u$-hermitian
element $x$ with $\Vert x \Vert \leq 1$, there exists a unitary
 $v$ in $X$ with $2x -v$ also unitary in $X$ (equivalently,
$x$ is the average of two unitaries).
The  condition 2)
also appears to be `characterizable' in many very different ways.  For example,
$X$ contains the product $vu$ in $C^*_{\rm e}(X)$ of any two
unitaries $u,v$ in $X$ if and only if the matrix $\left[
\begin{array}{ccl} 1 & u
\\ v & x
 \end{array} \right]$ is $\sqrt{2}$ times a unitary (characterized in
 Section 2) in $M_2(X)$ for some
$x \in X$.   Or there is a similar characterization using a $3
\times 3$ positivity condition as in \cite{Wa}.   We leave these to
the interested reader. The best such condition we have found to date
is as follows:

\begin{theorem} \label{main4}  A unital operator space $(X,1)$
(characterized above) is isomorphic to a $C^*$-algebra via a unital
complete isometry, if and only if {\rm 1)} for every $1$-hermitian
element $x$ with  $\Vert x \Vert = 1$, there exists a unitary
 $v$ in $X$ with $2x -v$ also unitary in $X$,
and {\rm 2)} \ for any unitary $v \in X$ and any $y \in {\rm
Ball}(X)$, there exists an element $z \in {\rm Ball}(X)$ with
$\left| \left| \left[
\begin{array}{ccl} t 1 & y \\ z & t v \end{array} \right] \right|
\right| \leq \sqrt{1 + t^2}$ for all $t \in \Rdb$.  It is also
equivalent to the conditions mentioned in  Theorem {\rm
\ref{main4prime}}.
 \end{theorem}

\begin{proof}   For the easy direction we may assume that $(X,1)$
is a unital $C^*$-algebra, and then by the $C^*$-identity the
asserted relation holds with $z = -v y^*$, and in fact in this case
the norm of the matrix is exactly $\sqrt{t^2 + \Vert y \Vert^2}$.
For the converse(s),  suppose that $1 = I_H \in X \subset B(H)$. If
the condition involving the infimum holds, then taking $v = 1$ we
see by Theorem \ref{main3} that $X$ is selfadjoint (that is, is an
operator system).  Also, for all $n \in \Ndb$ there is a $z_n \in X$
with
 $\left| \left|   \left[
\begin{array}{ccl}  n I & y \\ z_n &  n v
\end{array} \right]
 \right| \right|^2   \leq  n^2 + 1 + \frac{1}{n}$.
  Proceeding as in the proof of Theorem \ref{main3} we deduce that
$\Vert z_n + v y^* \Vert \leq 1/n + 1/n^2$, and so $v y^* \in X$.
That is, $v X = v X^* \subset X$.  Since the unitaries in $X$ are
spanning, $X$ is a $C^*$-subalgebra of $B(H)$.
\end{proof}

{\bf Remark.} \   The last theorem and its proof are still valid if
we replace ``$y \in {\rm Ball}(X)$'' with $y$ in the set of
unitaries in $X$.

\medskip

The next result gives a linear-metric `method' to retrieve a
forgotten product of any two elements, one of which is an isometry
or coisometry.

\begin{lemma} \label{mult}  Suppose that $I_H \in X \subset B(H)$,
and that $v$ is a coisometry (resp.\ isometry) on $H$ which lies in
$X$. If $y, z \in {\rm Ball}(X)$ then $z = - vy^*$ (resp.\ $y = -
z^* v$) in $B(H)$ if and only if $\left| \left| \left[
\begin{array}{ccl} t 1 & y \\ z & t v \end{array} \right] \right|
\right| \leq \sqrt{1 + t^2}$ for all $t \in \Rdb$.
\end{lemma}

\begin{proof}   This is similar to the proof of Theorem \ref{main4}.
If $a$ is the matrix in the lemma, then by applying states on
$M_2(B(H))$ to $a^* a$ or $a a^*$ we obtain as in that proof that
$\Vert z + vy^* \Vert = 0$ (resp.\ $\Vert y + z^* v \Vert = 0$).
\end{proof}

We note that Theorem \ref{main4} together with Sakai's theorem
immediately gives a linear-metric characterization of
$W^*$-algebras.  We mention an interesting related open question
concerning `dual operator systems'. We recall that an operator space
$X$ is a {\em dual operator space} if it is the operator space dual
of another operator space; and it is a fact that this is essentially
the same as saying that $X$ is a weak* closed subspace of some
$B(H)$.   For this, and for other aspects of the duality of operator
spaces we refer the reader to e.g.\ \cite[Section 1.4]{BLM}. As far
as we know, the analogous fact for operator systems, or for unital
operator spaces, is open: for example whether an operator system
which is also a dual operator space is isomorphic (in the obvious
sense) to a weak* closed operator system in $B(H)$, for some Hilbert
space $H$. We can offer the following `first step' in this
direction.

\begin{lemma} \label{dd} Suppose that $(X ,u)$ is a unital operator space
and suppose also that $X$ is also a dual Banach space.  Then
$\Delta^u$ and $\Delta^u_{\rm sa}$ are weak* closed, and the
involution on $\Delta^u$ is weak* continuous.
\end{lemma}

\begin{proof}   Suppose that  $(x_s)$ is a net in $\Delta^u_{\rm sa} \cap {\rm
Ball}(X)$, with $x_s \to x$ weak* in $X$.  Then $\Vert u + i t x_s
\Vert \leq \sqrt{1 + t^2}$.  Taking a limit with $s$ we see that $x
\in \Delta^u_{\rm sa}$.  Thus by the Krein-Smulian theorem,
$\Delta^u_{\rm sa}$ is weak* closed.    Next suppose that $(x_s + i
y_s)$ is a bounded net in $\Delta^u$, with limit $z$.  Here $x_s,
y_s \in \Delta^u_{\rm sa}$.  Then $(x_s)$ and $(y_s)$ are bounded
nets.  Suppose that a subnet $(x_{s_\lambda})$ converges weak* to
$x$ say.  Then $(y_{s_\lambda})$ has a subnet converging weak* to
$y$ say.  Replacing the nets by subnets, it is easy to see now that
$z = x + i y \in \Delta^u$.  So $\Delta^u$ is weak* closed. Finally,
suppose that $x_s + i y_s \to x + i y$ for $x, y \in \Delta^u_{\rm
sa}$.   The argument above shows that every weak* convergent subnet
of $(x_s)$ converges to $x$.  Thus $x_s \to x$ weak*.  Similarly,
$y_s \to y$ weak*, and so $(x_s + i y_s)^* = x_s - i y_s \to x - i
y$. By a variant of the Krein-Smulian theorem, the involution is
weak* continuous.
\end{proof}

\begin{corollary} \label{dd2}  If $X$ is an operator system which is also
a dual Banach space, then the involution on $X$ is weak* continuous.
\end{corollary}

It is easy to see that if $X$ is a dual operator space possessing a
weak* continuous conjugate linear involution $*$ for which $\Vert
[x^*_{ji}] \Vert_n = \Vert [x_{ij}] \Vert $ for all matrices
$[x_{ij}]$ with entries in $X$, then there exists a weak*
homeomorphic $*$-linear complete isometry from $X$ onto a weak*
closed selfadjoint subspace $W$ of some $B(H)$.  Indeed if $\varphi
: X \to B(H)$ is any weak* continuous complete isometry then the
function $x \mapsto \left[
\begin{array}{ccl} 0 & \varphi(x) \\  \varphi(x^*)^* & 0
\end{array} \right]$ does the trick.  This immediately gives an
abstract characterization of weak* closed selfadjoint subspaces of
$B(H)$.   By the last corollary, any operator system which is also a
dual operator space is at least one of the latter. However we are
unable to go further at this point.




\section{More on operator systems}

 The following facts will be useful
to us below. Given a unital operator space $(X,u)$, another way to
recapture the involution which is sometimes useful is as $u x^* u$,
the latter
 product and involution taken in a ternary envelope
${\mathcal T}(X)$. It is easy to see that $(X,u)$ is an operator system iff
$u X^* u \subset X$ within ${\mathcal T}(X)$, and in this case the expression
$u x^* u$ is independent of
the particular ternary envelope of $X$ chosen.
Also, we leave it to the reader to check that
the set of positive elements $\Delta^u_+$ in a unital operator space $(X,u)$
is precisely ${\mathfrak d}_u \cap X$, in the notation of
\cite{BN1}.   This is a very useful alternative description
of the positive elements in $X$.

\begin{example} \label{exC}  If $u$ is any unitary in a $C^*$-algebra or TRO $A$,
then $(A,u)$ is an operator system.  This follows by the
facts presented above Example \ref{exC}, since in this case $u A^* u \subset A$.
Moreover, any two unitaries $u, v \in A$ induce in some sense
the same operator system structure, since the map $T(x) = v u^* x$
on $A$ is a surjective complete isometry taking $u$ to $v$;
and hence $T$ is a complete
order isomorphism too, by e.g.\ 1.3.3 in \cite{BLM}.
\end{example}

The features in the last example fail badly for more general operator spaces.
It is easy to find operator spaces $X$ with unitaries     $u, v$
for which $(X,u)$ is not an operator system but
$(X,v)$ is; or for which they are both
operator systems but there exists no  surjective complete isometry
taking $u$ to $v$.   Moreover, the latter can be done
with $u$ and $v$ inducing the same involution on $X$.
We mention now explicit examples of these
phenomena.

\begin{example}  Let $X$ be the span of $1, f = e^{i \theta}$, and
$\bar{f}$,
in the continuous functions on the unit circle.  Then $(X,1)$ is an operator system,
but $(X,f)$ is not.  The latter is easily seen since the circle is the
Shilov boundary, so that $C^*_{\rm e}(X)$ is the space of continuous functions on the
circle, and this is a ternary envelope.   However $f X^* f \neq X$,
and so $(X,f)$ is not an operator system by the facts presented
above Example \ref{exC}.
\end{example}

\begin{example}  We describe a selfadjoint  space $X$ of continuous functions
on a compact topological space $K$ (equal to the Shilov boundary
of $X$), with $X$ containing constant functions,
and a unimodular continuous $g$ on $K$, such that
$(X,g)$ is an operator system with unchanged
involution, but
there exists no surjective isometric isomorphism $T : X \to X$ with
$T(1) = g$.

Let $K$ be the topological disjoint union of two
copies of the circle of unit radius centered at $(1,0)$,
and let $g$ be $1$ on the first circle and $-1$ on the other.
Let $f(z) = z$ for any $z$ in either circle.  Let $X = {\rm Span}(\{1, g, f,\bar{f}\})$.
Then the
reader can check, using Proposition \ref{twouv}  and Corollary  \ref{okn}  below,
and the fact that a $*$-isomorphism of $C(K)$ is `composition with
a homeomorphism' $\tau : K \to K$,  that
 $X$ has all the
properties described in the last paragraph.
\end{example}

\begin{proposition} \label{twouv}  Let $X$ be a unital operator space
viewed within its $C^*$-envelope $A = C^*_{\rm e}(X)$.
Suppose that $v$ is a unitary in $X$.
  Then there exists a
surjective complete isometry $T : X \to X$ with $T(1) = v$
if and only if there is a $*$-isomorphism $\theta : A \to A$
such that $v^* X =  \theta(X)$.
\end{proposition}

 \begin{proof}    For the one direction simply set $T = v \theta(\cdot)$.
Conversely, suppose that  $T : X \to X$ is as stated.
Since the
$C^*$-envelope is a ternary envelope, by universal properties
of the ternary envelope we may extend
$T$ to a surjective complete isometry
$\tilde{T} : C^*_{\rm e}(X) \to C^*_{\rm e}(X)$.  Then
$\tilde{T}$ is a ternary morphism by \cite[Corollary 4.4.6]{BLM}, and
$\theta = v^* \tilde{T}(\cdot)$ is easily seen to be a $*$-isomorphism of
$A$ onto itself.   The rest is obvious.   \end{proof}


{\bf Remark.} Lemma  \ref{mult}
supplies a `linear-metric' sufficient condition for the existence of
a complete isometry $T$ with $T(u) = v$ as in the last result.
Namely, suppose that $(X,u)$ is an operator system, and that $v$ is
another unitary in $X$. Without loss of generality we may take $X =
X^* \subset A = C^*_{\rm e}(X)$ with $1_A = u$. With respect to the
structure in $A$, Lemma  \ref{mult} shows that $v X^* \subset X$ if
and only if any $y \in {\rm Ball}(X)$, there exists an element $z
\in {\rm Ball}(X)$ with $\left| \left|   \left[
\begin{array}{ccl} t 1 & y \\ z & t v
\end{array} \right]  \right| \right| \leq \sqrt{t^2 + 1}$ for all $t
\in \Rdb$.   By symmetry, $X^* v \subset X$ iff a similar condition
to that in the last line holds, but with the $1$-$2$ and the $2$-$1$
entries switched in the matrix. It is easy to see that both of the
conditions in the last two sentences hold simultaneously iff  $v X^*
= X$, that is, iff $v X = X$. In turn, if the latter holds, then
clearly the map $T(x) = vu^*x$ is a complete isometry from $X$ onto
$X$ with $T(u) = v$. Of course the latter forces $(X,v)$ to be an
operator system (and $T$ to be `$*$-linear').



\begin{proposition} \label{twou}
If $X$ is an operator space
with unitaries $u$ and $v$ such that
$(X,u)$ and $(X,v)$ are operator systems,
then the involutions on these two systems are the same
if and only if $u^* v$ is in the center of $Z^* Z$
and equals $v^* u$; where $Z = {\mathcal T}(X)$.
\end{proposition}

 \begin{proof}  To say that the involutions are the same
is to say that $u x^* u = v x^* v$ for all $x \in X$.
Setting $x = u$ gives $u = v u^* v$, so that $v^* u = u^* v$.
Moreover, it is simple algebra to check that
$x^* y u^* v = u^* v x^* y$ for $x, y \in X$.
Since spans of products of terms of the form $x^* y$ for $x, y \in X$
are dense in $Z^* Z$, we deduce that
$u^* v$ is in the center of $Z^* Z$.

Conversely, suppose that $u^* v = v^* u$ is in the center of $Z^* Z$.
Then for all $x \in Z$,
we have $ux^*u=vv^*ux^*u=vx^*uv^*u=vx^*uu^*v=vx^*v$.
      \end{proof}

\begin{corollary}  \label{okn}  Let $v$ be a unitary  in an operator
subsystem  $X \subset B(H)$.  Then $v \in X_{\rm sa}$  and $v$ is in the center of $C^*_{\rm e}(X)$,
if and only if $(X,v)$ is an
operator system, and the involution associated with $v$
equals the original involution.
\end{corollary}

 \begin{proof}    The one direction follows immediately from
Proposition \ref{twou}.  For the other, if $v = v^*$ and $v$ is in
the center of $C^*_{\rm e}(X)$, then
clearly $(X,v)$ is an operator system, and it has unchanged
involution.
      \end{proof}

{\bf Remark.}  Of course saying that the involutions associated with
two unitaries $u$ and $v$ coincide, is equivalent to saying that
every $u$-hermitian is $v$-hermitian.   This is assuming that
$(X,u)$ is an operator system.

%


\medskip


In the remainder of the section, we mention a connection to the
famous characterization due to Choi and Effros of operator systems
\cite{CE} in terms of a given cone  in the space.  We will assume
throughout that we have a fixed cone ${\mathfrak c}$ in $X$, and
that this cone spans $X$ (although this also often follows as a
consequence of some of the conditions imposed below). We allow two
variants of the theory: depending on whether or not we are assuming
the existence of a given fixed involution $*$ on $X$. If the latter
holds, we will assume further that $x = x^*$ for all $x \in
{\mathfrak c}$.

\begin{definition} \label{oos}
By an {\em ordered operator space} below we will mean a pair
$(X,{\mathfrak c})$ consisting of an operator space and a cone
${\mathfrak c}$ in $X$, such that there exists a complete isometry
of $X$ into a $C^*$-algebra $A$  taking ${\mathfrak c}$ into $A_+$.
\end{definition}

\begin{proposition} \label{ord1}
Suppose that $(X,{\mathfrak c})$ is an ordered operator space, and
that $u$ is a unitary in $X$ contained in ${\mathfrak c}$. If
 ${\mathfrak c}$ spans $X$ then  $(X,u)$ is an
operator system. Moreover, in the `involutive variant'
of the theory (mentioned above Definition {\rm \ref{oos}}),
the involution induced by $u$
equals the original involution $*$.
 \end{proposition}

\begin{proof}
We use notation and facts from \cite{BN1}.  Consider the ordered
ternary envelope of $X$, whose positive cone is a natural cone
${\mathfrak d}_v$ given by an open tripotent $v$.  Also, ${\mathfrak
c} \subset {\mathfrak d}_v$. Since $u \in {\mathfrak c} \subset
{\mathfrak d}_v$, and $u$ is unitary, it follows that $u = v$. Thus
${\mathfrak c} \subset {\mathfrak d}_u \cap X = \Delta^u_+$.  Since
${\mathfrak c}$ is spanning, this implies by Corollary \ref{elm}
that $(X,u)$ is an operator system.
In the `involutive variant', notice that $u x^* u = x = x^*$ for $x
\in {\mathfrak c}$, and hence for $x \in X$ since ${\mathfrak c}$ is
spanning. This also shows that $u$ is central in the sense of
\cite{BW}  in the ordered ternary envelope of $X$.
\end{proof}




Next, we seek conditions which imply that $(X,u)$ is an operator system
whose cone is precisely ${\mathfrak c}$.   If we also had cones in $M_n(X)$, then
necessary and sufficient conditions for this may be found in
 the famous characterization due to Choi and Effros
of operator systems \cite{CE,Pau}. The most prominent of these
conditions is the existence of a `matricial order unit'. We show
here that the following weaker condition suffices:

\begin{definition}   We say that $u$ is a {\em norm-order unit}
for ${\mathfrak c}$ if $u \in {\mathfrak c}$ and
for every $x \in X_{\rm sa}$ we have $\Vert x \Vert u - x \in {\mathfrak c}$.
\end{definition}

In the `non-involutive space variant' of the theory (see the discussion
above Definition \ref{oos}), we replace
$X_{\rm sa}$ here by the $u$-hermitians on $X$.

\begin{proposition} \label{ord}
Suppose that $(X,{\mathfrak c})$ is an ordered operator space, and
that $u$ is a unitary in $X$ which is a norm-order unit for
${\mathfrak c}$. Then  $(X,u)$ is an operator system whose positive
cone $\Delta^u_+$ is ${\mathfrak c}$.
   \end{proposition}

\begin{proof}
We saw in Proposition \ref{ord1} that  $(X,u)$ is an operator
system, ${\mathfrak c} \subset {\mathfrak d}_u$, and that in the
`involutive space variant' of the theory the involution is
unambiguous.   If $x \in {\mathfrak d}_u \cap X = \Delta^u_+$ then $x$
is $u$-positive.  If $t = \Vert x \Vert$, then  $t u - x \in
{\mathfrak c}$. Hence $t u - x = c$ for some $c \in {\mathfrak c}$.
Viewed within the $C^*$-algebra $Z''_2(u)$
(notation as in \cite{BN1}), we have $\Vert c \Vert =
\Vert t u - x \Vert \leq t$, since $t u - x \leq t u$ in $Z''_2(u)$.
Thus $x = t u - c = (t - \Vert c \Vert )u + (\Vert c \Vert u - c)
\in {\mathfrak c}$. So $\Delta^u_+ = {\mathfrak c}$.
\end{proof}


\begin{corollary} \label{ordu}  Suppose that $(X,u)$ is an operator system and
that ${\mathfrak c}$ is a subcone of $\Delta^u_+$, such that
$u$ is a norm-order unit for ${\mathfrak c}$. Then $\Delta^u_+ =
{\mathfrak c}$.
 \end{corollary}

 \begin{proof}
Clearly $(X,{\mathfrak c})$ is an ordered operator space, and we are in the
situation of Proposition \ref{ord}.
\end{proof}


{\bf Remark.}   In the previous context, the range tripotent of an
order unit can be shown to be unitary. Hence an order unit which is
a partial isometry is unitary.

\section{Function spaces}

By a {\em function space} we will mean a closed subspace of a
commutative $C^*$-algebra.   Abstractly, these are just the Banach
spaces, or equivalently the operator spaces which have the `Min'
structure (eg.\ see 1.2.21 in \cite{BLM}). By a {\em selfadjoint
function space} we will mean a closed selfadjoint subspace of a
commutative $C^*$-algebra. By a {\em unitary in} $X$, where $X$ is a
Banach space, we will mean an element $g \in X$ such that there
exists a linear isometry $T : X \to C(K)$, for a compact Hausdorff
space $K$, with $T(g) = 1$.   We shall show below that there is no
conflict with the earlier definition of a unitary in $X$.  There is
however a conflict with the notation in \cite{Jar}.  Indeed, if $K$
is the Shilov boundary of a unital function space $X$, then
unitaries in $X$ in our sense are just the elements in $X$ which are
unimodular on $K$.
 A {\em unital function space} is a pair $(X,g)$ where
$g$ is a unitary in the Banach space $X$.  By replacing $C(K)$ by
the $C^*$-algebra generated by $T(X)$ we may assume that $T(X)$
separates points of $K$, a common assumption in the function theory
literature. Indeed, unital function spaces may be taken as the
`basic setting' for the presentation of what some might call the
`classical Shilov boundary', as is explained in Section 4.1 of
\cite{BLM}.   Although it is not very difficult, we are not aware of
any abstract characterization of unital function spaces in the
literature until now.

\begin{theorem}  \label{fsp}  Let $g$ be an element in a Banach space $X$.
The following are equivalent:
\begin{itemize}
\item [(i)]  $g$ is a unitary in $X$ in the sense above.
\item [(ii)]   $g$ is a unitary in ${\rm Min}(X)$
in the sense of the introduction of our paper.
\item [(iii)]  $\sup \{ \Vert s f + t g \Vert :
s, t \in \Cdb, |s|^2 + |t|^2 = 1 \} = \sqrt{2}$ for
every $f \in X$ with $\Vert f  \Vert = 1$.
\end{itemize}
Thus $(X,g)$ is  a unital function space if and only if {\rm (iii)} holds.
\end{theorem}

\begin{proof}   (i) $\Rightarrow$ (ii) \  This follows because
every isometry between minimal operator spaces
is a complete isometry (see e.g.\ 1.2.21 in \cite{BLM}).

(ii) $\Rightarrow$ (i) \ By 4.2.11 in \cite{BLM} the injective envelope, and hence the
ternary envelope $Z$, of ${\rm Min}(X)$, is a minimal operator space.
Hence $Z^* Z$ is commutative by  \cite[Proposition 8.6.5]{BLM}.
Since $gX$ is unitary in $Z$ by Lemma
\ref{note} (iii),  the map $x \mapsto g^* x$ is a `unital isometry'
from $X$ into the commutative $C^*$-algebra $Z^* Z$.

(iii) $\Rightarrow$ (ii) \ One may prove that (iii) $\Rightarrow$
(i) using classical techniques, but instead we will deduce it from
our noncommutative results.  It is well known (and an easy
exercise), that the norm of $[ f \; \; g ]$, as a row with entries
in ${\rm Min}(X)$, is $\sup \{ \Vert s f + t g \Vert : s, t \in
\Cdb, |s|^2 + |t|^2 = 1 \}$, and a similar statement holds for
columns.  Following the proof of Theorem \ref{uosp} we see that the
operator $L$ there is an isometry. Hence it is a complete isometry
by facts mentioned in the last two paragraphs, and as in the other
proof this implies $u u^* = 1$ in $Z^* Z$.   Similarly, $u^* u = 1$.
That is, $u$ is a unitary in $Z$, and we can apply Lemma \ref{note}.

That (i) implies (iii) is left as an exercise.
 \end{proof}

By a {\em function system} we will mean a closed selfadjoint subspace of $C(K)$,
for compact $K$, containing constant functions.  There is
an obvious `abstract definition': $(X,g)$ is a function system if there
exists an isometry $T : X \to C(K)$ with $T(g) = 1$ and $T(X)$ selfadjoint.
 It is easy to see that this
 is equivalent to saying that   $({\rm Min}(X),g)$ is an operator system.
See also \cite{PT}. Many of the results in earlier sections
concerning operator systems have `function system' analogues.  For
example, function systems are the obviously the unital function
spaces $(X,g)$ which are spanned by their $g$-hermitians.  The
following is another characterization of function systems,  which
also improves on Corollary \ref{okn} in our present situation:

\begin{theorem}  \label{okn2}  Let $X$ be a selfadjoint function space (defined above).
If  $v$ is any unitary in $X$ with $v = v^*$
then $(X,v)$ is a function system
and the involution associated with $v$ equals the original involution.
\end{theorem}

 \begin{proof}
Let $X$ be a selfadjoint subspace of $C(K)$, for compact $K$.
The TRO generated by $X$ in $C(K)$ is a $*$-subTRO of $C(K)$, and it
follows that the `ternary $*$-envelope' $Z$ of $X$ (see \cite{BNW}) is `commutative': that
is $x y = yx$ for all $x, y \in Z$.    Also $C = v^* Z$ is a commutative unital $C^*$-algebra,
and $T(x) = v^* x$ is an isometric `unital' map into $C$.  Since $X$ is a
a selfadjoint subspace of $Z$  it is easy to see that $T$ is `$*$-linear',
and the rest is obvious.  Alternatively,
note that $v x^* v = x^* v v^* = x^*$ for $x \in X$.
\end{proof}

The class of selfadjoint function spaces (not necessarily with a unit),
which  appear in the last result,
 may be characterized abstractly, although again we have
not seen this in the literature.   For example, we have the
following characterization in terms of `selfadjoint functionals', by
which we mean that $\varphi(x^*) = \overline{\varphi(x)}$ for $x \in
X$.  We remark that this result unfortunately violates our principle
of not using maps or functionals on the space; perhaps a better
characterization will be forthcoming.

\begin{proposition} Selfadjoint function spaces are precisely the Banach spaces
$X$ with an involution $*$ having the property that every extreme
point of ${\rm Ball}(X^*)$ is a scalar multiple of a selfadjoint
functional.
\end{proposition}

\begin{proof}  Any $X$ with the announced property is clearly
$*$-linearly isometric to a selfadjoint subspace of the continuous
functions on the weak* compact set of selfadjoint elements in Ball$(X^*)$.
For the converse, we may suppose that
 we have a selfadjoint subspace $X \subset C(K)$, with $K$
compact (by replacing the commutative $C^*$-algebra with its
unitization). By the routine Krein-Milman type argument, any extreme
point for Ball$(X^*)$ is the restriction of an extreme point for
Ball$(C(K)^*)$. Hence it is of the form $\chi \delta_w$, where $\chi
\in \Cdb$ and $\delta_w$ is point evaluation at $w \in K$. It is
clear that $\delta_w$ is selfadjoint.
  \end{proof}

\medskip

{\bf Closing remark.}   Having an abstract characterization of a class
of objects is often useful in order to show that the class is closed
under the usual shopping list of `constructions', such as direct sums,
certain quotients, tensor products,  ultraproducts, interpolation,
etc.  In our case one may certainly do this, but we will not
do so here, for the reason that all of these can seemingly be done
without appealing to our new characterizations.   For example, as C. K. Ng has
suggested to us privately, one may prove their result from
\cite{HN} about quotients by $M$-ideals using our criteria too.
 We mention a third
way to prove this result: if $X$ is a unital operator space (resp.\
operator system) sitting in its $C^*$-envelope $A$, then by basic
facts about $M$-summands from e.g.\ \cite[Section 4.8]{BLM} and
references therein, we may view a complete $M$-projection on $X$ as
a projection $p$ in the center of $A$ which is also in $X$.  Thus $p
X$ is a unital subspace (resp.\ operator subsystem) of $p A p$. This
does the $M$-summand case, and the $M$-ideal case follows by the
idea in \cite{HN} of going to the second dual.

\medskip

{\bf Acknowledgement.}  The last section of the paper was written 
shortly after the first version of our paper was distributed.
We have also added several references (see e.g.\ the discussion at the
end of the introduction).
  




\end{document}